\documentclass[reqno, 12pt]{amsart}
\usepackage{enumerate}
\usepackage{amssymb,url,color}
\usepackage{hyperref}
\usepackage{amsmath,amsthm}

\allowdisplaybreaks[4]

\newtheorem{thm}{Theorem}[section]

\newtheorem{prop}{Proposition}[section]
\newtheorem{lem}{Lemma}[section]
\newtheorem{rem}{Remark}[section]

\theoremstyle{definition}

\numberwithin{equation}{section}
\newcommand{\pp}{\mathbb{P}}
\newcommand{\nn}{\mathbb {N}}
\newcommand{\ee}{\mathbb{E}}

\newcommand{\FF}{\mathcal{F}}

\newcommand{\rr}{\mathbb{R}}
\newcommand{\ii}{\mathbb{I}}

\def\beq{\begin{equation}}
\def\deq{\end{equation}}
\def\dsp{\displaystyle}
\allowdisplaybreaks 

\bfseries\rmfamily 

\begin{document}
\title[Large deviation inequalities for urn model]
{Large deviation inequalities for the nonlinear unbalanced urn model}
\thanks{This work is supported by National Natural Science Foundation of China (NSFC-11971154).}

\author[J. N. Shi]{Jianan Shi}
\address[J. N. Shi]{College of Mathematics and Information Science, Henan Normal University, Henan Province, 453007, China.}
\email{\href{mailto: J. N. Shi
<jiananshi2022@126.com>}{jiananshi2022@126.com}}

\author[Z. H. Yu]{Zhenhong Yu}
\address[Z. H. Yu]{College of Mathematics and Information Science, Henan Normal University, Henan Province, 453007, China.}
\email{\href{mailto: Z. H. Yu
<zhenhongyu2022@126.com>}{zhenhongyu2022@126.com}}

\author[Y. Miao]{Yu Miao}
\address[Y. Miao]{College of Mathematics and Information Science, Henan Normal University, Henan Province, 453007, China; Henan Engineering Laboratory for Big Data Statistical Analysis and Optimal Control, Henan Normal University, Henan Province, 453007, China.} \email{\href{mailto: Y. Miao
<yumiao728@gmail.com>}{yumiao728@gmail.com}; \href{mailto: Y. Miao <yumiao728@126.com>}{yumiao728@126.com}}

\begin{abstract}
In the present paper, we consider the two-color nonlinear unbalanced urn model, under a drawing rule reinforced by an $\rr^+$-valued concave function and an unbalanced replacement matrix.
The large deviation inequalities for the nonlinear unbalanced urn model are established by using the stochastic approximation theory. As an auxiliary theory, we give a specific large deviation inequality for a general stochastic approximation algorithm.
\end{abstract}


\maketitle
\section{Introduction}
The urn model is a widely used mathematical tool across various disciplines such as algorithms, genetics, epidemiology, physics, engineering, economics, and social networks, serving to model diverse evolutionary processes. The earliest urn model could date back to the P\'olya urn model, which is to draw a ball from an
urn containing initially two balls of different colors and then put it back in the urn along with another of the same color as the sampled one. Eggenberger and P\'olya \cite{E-P} studied the opposite extraction policy, i.e., returning the drawn ball along with another of the opposite color. Friedman \cite{F-B} further considered the following problem. An urn contains $W$ white and $B$ black balls. A ball is drawn at random and then $1+\alpha$ additional balls of the same color and $\beta$ additional balls of the opposite color are added. After this procedure has been repeated $n$ times what is the probability distribution of the number of white balls? The problem was reduced to a difference-differential equation.
 Johnson and Kotz \cite{N-S} introduced many urn models which involve stochastic replacement of balls and the irreplaceable occupancy problems, and gave some applications to various fields of science in the framework of statistics and probability. As a compact sequel to the book, Kotz and Balakrishnan \cite{S-B} surveyed some key developments that have occurred pertaining to urn models in probabilistic, statistical and biological literatures. Mahmoud \cite{H} presented an extensive treatment of urn models. Except for a few simple motivating examples, the focus is on the P\'olya urn models (or schemes), describing urns containing a finite number of coloured balls that are repeatedly selected and possibly replaced in the containers according to a pre-determined strategy.

 In the P\'olya urn models, there is a fundamental commonality across most variants. Initially, the urn contains $W_0$ white balls and $B_0$ black balls. The evolution of the urn proceeds in discrete time steps. During each step, a ball is randomly chosen from the urn. After noting its color, the ball is returned to the urn. According to the observed color of the ball, balls are
added due to the following rules. If a white ball is drawn, $a$ white balls and $b$ black balls are added, and if a black ball is drawn, $c$ white balls and $d$ black balls are added.  The values $a, b, c,$ and $d$ are integers and the urn model is specified by the $2\times2$ replacement matrix
$$
A:= \begin{pmatrix}
a & c\\
b & d
\end{pmatrix}
.
$$
The most fundamental urn model is the original P\'olya-Eggenberger urn model, associated with the ball replacement matrix
 $$
A:= \begin{pmatrix}
1 & 0\\
0 & 1
\end{pmatrix}
.
$$
The P\'olya-Eggenberger urn is a balanced urn model, i.e.,  the total number of added
balls is constant and so independent of the observed color. Chen and Wei \cite{C-W05} introduced a urn model in which more than one ball is drawn randomly each time, generalizing the P\'olya-Eggenberger urn model. Chen and Kuba \cite{C-K13} studied the above general urn model and provided exact expressions for the expectation and the variance of the number of white balls after $n$ draws, and determine the structure of higher moment. Kuba and Mahmoud \cite{K-M} considered a class of balanced urn schemes on balls of two colors (say white and black). At each drawing, a sample of size $m\ge 1$ is taken out from the urn, and ball addition rules are applied. They obtained central limit theorems for small and critical-index urns and prove almost-sure convergence for triangular and large-index urns. These urn models are linear in the following sense:
$$
\ee(W_n | \FF_{n-1}) = \varphi_n W_{n-1} + \nu_n,
$$
where $(\FF_n)_{n\ge 0}$ is the $\sigma$-algebra generated by the first $n$ steps, $W_n$ represents the number of white balls at time $n$ in an urn initially containing $W_0$ white balls and $B_0$ black balls, $\varphi_n$  and $\nu_n$ are two sequences dependent solely on the number of draws.
Aguech et al. \cite{A-1} and Aguech and Selmi \cite{A-2} studied a linear unbalanced urn model with multiple drawing. At each discrete time step $n$, one draws $m$ balls at random from an urn containing white and blue balls. The replacement of the balls follows either opposite or self-reinforcing rules. The authors obtained a strong law of large numbers and a central limit theorem for $W_n$, the number of white balls after $n$ draws.

Laruelle and Pag\'es \cite{L-P} presented a class of nonlinear urn models, where the drawing rule is neither uniform nor hypergeometric, but reinforced by a function $f:\rr^+\rightarrow \rr^+$ such that, the probability of drawing a ball of type $i$ obeys the following rule:
$$
\pp(X_{n+1}=e_i|\FF_n)=\frac{f(Y^i_n)}{\sum_{j=1}^d f(Y^j_n)},\ n\ge0,
$$
where $Y^i_n$ represents the count of balls of color $i$ in a $d$-colored urn at time $n$ and $(\FF_n)_{n\ge 0}$
is the $\sigma$-algebra generated by the first $n$ draws, and $X_n:(\Omega,\mathcal{A}, \pp)\rightarrow \{e_1, \ldots, e_d\}$ denotes the type of the drawn ball at time $n$. This kind of drawing rules is called nonlinear drawing rule. Note that, if $f\equiv id_{\rr^+}$ the model boils down to a classic single $d$-colored urn. Idriss \cite{IS} studied a nonlinear unbalanced urn model, which adopted the model proposed in Laruelle and Pag\'es \cite{L-P} and used the non-balance assumption introduced in Aguech et al. \cite{A-1} and Aguech and Selmi \cite{A-2}.

 Recently stochastic approximation methods are used to obtain asymptotic results on the urn models.
   The theory of stochastic approximation was initially introduced by Robbins and Monro \cite{R-M} to solve equations of the form $M(x)=\alpha$. Renlund \cite{R-1} and Renlund \cite{R-2} developed a stochastic approximation algorithm tailored for generalized P\'olya urn models and derived asymptotic results within the framework of stochastic approximation theory. Idriss \cite{IS} obtained the central limit theorem and the almost sure convergence of the nonlinear unbalanced urn model by using the methods in Renlund \cite{R-1} and Renlund \cite{R-2}.
   In the present paper, we continue to consider the nonlinear unbalanced urn model, which was derived from Idriss \cite{IS}, and obtain a large deviation inequality for the urn model. As an auxiliary theory, we give a specific large deviation inequality for a general stochastic approximation algorithm.
 Our proof methodology was influenced by Koml\'os and R\'ev\'esz \cite{K-R}, who established the convergence rate of the stochastic approximation algorithm. In Section 2, we state the nonlinear unbalanced urn model and give the main results, and some preliminary lemmas and the proofs of these main results are given in Section 3.

\section{Main results}

\subsection{Nonlinear unbalanced urn model}
 We consider an urn containing two types of balls: type $1$ and type $2$. All random variables pertinent to this model are defined on a probability space $(\Omega,\mathcal{A}, \pp)$.
Initially, the urn contains $Y^1_0\in \nn^+$ balls of type $1$ and $Y^2_0\in \nn^+$ balls of type $2$, says $T_0=Y^1_0+Y^2_0$ as initial total of balls.
If the drawn ball is of type $j\in \{1,2\}$, the urn is updated by adding $H^{ij}$ balls for every $i\in \{1,2\}$. In order to make this process to be
repeated safely, we require the urn to be tenable, meaning it never dies, for this we set the
replacement matrix $H$ with non negative entries as follows:
\beq\label{m-1}
H:= \begin{pmatrix}
H^{11} & H^{12}\\
H^{21} & H^{22}
\end{pmatrix}
,\ H^{ij}\ge 0.
\deq
 Let us define $H_1:=H^{11}+H^{21}$ and $H_2:=H^{12}+H^{22}$, and assume $H_1\neq H_2$, which indicates that the replacement matrix is no longer balanced.
    Let $(Y_n)=(Y^1_n,Y^2_n)^{'} \in \nn^2 \backslash \{0\},$ denote the composition of the urn at stage $n$, where $T_n:=Y^1_n+Y^2_n$ represents the total number of balls at the stage $n$.
 Let $(\FF_n)_{n\ge 0}$ be the $\sigma$-algebra generated by the first $n$ steps, and the urn model is represented by the following normalized skewed drawing rule:
\beq\label{1}
\pp (X_{n+1}=e_i | \FF_n)=\frac{f(\frac{Y^i}{T_n})}{f(\frac{Y^1}{T_n})+f(\frac{Y^2}{T_n})},\ \ i\in \{1,2\},
\deq
 where $f:[0,1]\rightarrow \rr^+$ is a skewing function satisfying some the properties and $X_n:(\Omega,\mathcal{A}, \pp)\rightarrow \{e_1, e_2\}$ represents the type of the drawn ball at time $n$.

 We use the following stochastic recursion to describe the evolution of the urn model:
\beq\label{2}
Y_{n+1} = Y_n+H X_{n+1}, \ \ \ Y_0\in \nn^2\ \backslash \{0\},
\deq
which means
$$
Y^1_{n+1}=Y^1_n+\sum_{j=1}^2 H^{1j}\ii_{\{X_{n+1}=e_j\}}
$$
and
$$
Y^2_{n+1}=Y^2_n+\sum_{j=1}^2 H^{2j}\ii_{\{X_{n+1}=e_j\}}.
$$
The total number of balls is given by
\beq\label{3}
T_{n+1}= T_n+\sum_{i=1}^2 \sum_{j=1}^2 H^{ij} \ii_{\{X_{n+1}=e_j\}}.
\deq
Here we remark that if $H_1=H_2=1$, then we get
$$
\aligned
T_{n+1}=& T_n+\sum_{i=1}^2 \sum_{j=1}^2 H^{ij} \ii_{\{X_{n+1}=e_j\}}\\
=&T_n+H_1\ii_{\{X_{n+1}=e_1\}}+H_2\ii_{\{X_{n+1}=e_2\}}\\
=&T_n+1\\
\vdots&\\
=&T_0+n.
\endaligned
$$
\subsection{Assumptions and some known results}
In order to study our results, the following conditions are needed:
\begin{enumerate}[($C_1$)]
 \item Let the function $f: [0,1]\to \rr^+$ be continuous, non-decreasing, differentiable, $f(0)=0$, $f(1)=1$, $f>0$ on $(0,1]$, with finite right derivative at $0$ and left derivative at $1$. We will denote by $f^{'}(0)$ and $f^{'}(1)$ the right and left derivatives of $f$ at $0$ and $1$.

 \item The matrix $H$ is not balanced, which means that $H_1\ne H_2$.

 \item The urn scheme is tenable. It means that as long as the process is running, the urn is
never empty.

\item The entries of the matrix $H$ are neither all null nor negative. Moreover, the entries
$H^{12}$ and $H^{21}$ are strictly positive.
    \end{enumerate}

In order to obtain our main result, let $Z_n:=\frac{Y^1_n}{T_n}$ represent the proportion of balls of type 1 in the urn. So for $Z_0\in(0,1]$ we have
\begin{align}\label{a-1}
Z_{n+1}-Z_n&=\frac{Y_{n+1}^1}{T_{n+1}}-\frac{Y_{n}^1}{T_{n}}\nonumber\\
&=\frac{1}{T_{n+1}}\left(Y_n^1+\sum_{j=1}^2 H^{1j}\ii_{\{X_{n+1}=e_j\}}-\left(T_n+\sum_{i=1}^2 \sum_{j=1}^2 H^{ij} \ii_{\{X_{n+1}=e_j\}}\right)\frac{Y_{n}^1}{T_{n}}\right)\nonumber\\
&=\frac{1}{T_{n+1}}\left(\sum_{j=1}^2 H^{1j}\ii_{\{X_{n+1}=e_j\}}-\left(\sum_{i=1}^2 \sum_{j=1}^2 H^{ij} \ii_{\{X_{n+1}=e_j\}}\right)Z_n\right)\nonumber\\
&=\frac{1}{T_{n+1}}\left(\left(H^{11}-Z_nH_1 \right)\ii_{\{X_{n+1}=e_1\}}+\left(H^{12}-Z_nH_2 \right)\ii_{\{X_{n+1}=e_2\}}\right)\nonumber\\
&=\frac{L_{n+1}}{T_{n+1}},
\end{align}
where
$$
L_{n+1}=(H^{11}-Z_nH_1)\ii_{\{X_{n+1}=e_1\}}+(H^{12}-Z_nH_2)\ii_{\{X_{n+1}=e_2\}}.
$$
It is easy to check that
\begin{align*}
\ee[L_{n+1}|\FF_n]&=\ee\left[(H^{11}-Z_nH_1)\ii_{\{X_{n+1}=e_1\}}+(H^{12}-Z_nH_2)\ii_{\{X_{n+1}=e_2\}}|\FF_n\right]\\
&=\frac{(H^{11}-Z_nH_1)f(Z_n)+(H^{12}-Z_nH_2)f(1-Z_n)}{f(Z_n)+f(1-Z_n)},
\end{align*}
then we have
\begin{align*}
Z_{n+1}-Z_n=&\frac{\ee[L_{n+1}|\FF_n]+L_{n+1}-\ee[L_{n+1}|\FF_n]}{T_{n+1}}\\
&=\frac{1}{T_{n+1}} (h(Z_n)+\Delta M_{n+1}),
\end{align*}
where
\beq\label{b-5}
\Delta M_{n+1}=L_{n+1}-\ee[L_{n+1}| \FF_n],
\deq
and
\beq\label{c-1}
h(y)=\frac{(H^{11}-yH_1)f(y)+(H^{12}-yH_2)f(1-y)}{f(y)+f(1-y)},\ \ y\in[0,1].
\deq

Let $h$ be as in (\ref{c-1}), and define the sets $\Theta^{*}$ and $I^{*}$ as follows:
$$
\Theta^{*}:=\{y\in[0,1]; h(y)=0\}
$$
and
 \beq\label{I}
 I^{*}:=\left\{\min\left(\frac{H^{11}}{H_1}, \frac{H^{12}}{H_2} \right), \ \max\left(\frac{H^{11}}{H_1}, \frac{H^{12}}{H_2} \right)\right\}.
\deq
$\Theta^{*}$ denotes the set of zeros of the function $h$ and may be called the set of equilibrium points.
From Proposition 2.1 in \cite{IS}, we know that the set $\Theta^{*}$ is non-empty, and if $y^{*}\in \Theta$, then $y^{*}\in I^{*}$.

\begin{prop}\label{prop-i1}\cite[Proposition 2.2]{IS} Under the conditions $(C_1)$, $(C_2)$, $(C_3)$ and $(C_4)$, then we have the following claims.

\begin{enumerate}[$(1)$]
 \item If $\dsp \frac{H^{11}}{H_1}\le  \frac{H^{12}}{H_2}$, then $h$ has a unique solution in $I^{*}$.

 \vskip10pt

 \item If $\dsp \frac{H^{11}}{H_1}> \frac{H^{12}}{H_2}$ and $f$ is concave, then $h$ has a unique solution in $I^{*}$.
    \end{enumerate}
\end{prop}

The following properties, which are proved in \cite[Proof of Theorem 2.1]{IS}, will be used to obtain the large deviation inequality of $(Z_n)_{n\ge 0}$.

\begin{prop}\label{prop1}
Under the conditions $(C_1)$, $(C_2)$, $(C_3)$ and $(C_4)$, then the following claims hold.
\begin{enumerate}[$(1)$]
 \item For any $n$, we have
\beq\label{b-1}
\frac{1}{2n\max\{H_1, H_2\}}\leq \frac{1}{T_n}\leq\frac{1}{2n\min\{H_1, H_2\}}.
\deq

\item For any $y\in[0,1]$, we have
\beq\label{b-2}
|h(y)|\leq2(H_1+H_2).
\deq

\item For any $n$, we have
\beq\label{b-3}
\mid\Delta M_{n+1}\mid\le 4(H_1+H_2).
\deq

\item There exists a positive constant $K=K(H_1, H_2)$, such that
\beq\label{b-6}
\left|\ee\left(\frac{\Delta M_{n+1}}{T_{n+1}}\  \Bigg|\ \FF_n\right)\right|\le \frac{K}{T^2_n}.
\deq
    \end{enumerate}
\end{prop}

Let us recall the definition of stable point. A zero $y^{*}$ of $h$ is called stable if $h$ is differentiable and $h^{'}(y^{*})<0$. Idriss \cite{IS} proved the following claim \cite[Proposition 2.3]{IS}:
\beq\label{l}
\text{\it if $h$ has a unique equilibrium point then it is stable.}
\deq
From the proof of the claim (\ref{l}), it follows that there exists $\varepsilon>0$ small enough such that
\beq\label{l1}
h(x)>0\ \text{for}\ x\in(y^{*}-\varepsilon, y^{*})\ \ \ \text{and}\ \ \ h(x)<0, \ \text{for}\  x\in(y^{*}, y^{*}+\varepsilon),
\deq
which implies that
$$
(x-y^{*})h(x)<0,\ \ x\in(y^{*}-\varepsilon, y^{*}+\varepsilon).
$$

\subsection{Main results} Firstly we give the large deviation inequality for the nonlinear unbalanced urn model.

\begin{thm}\label{thm2-1}
Under the conditions $(C_1)$, $(C_2)$, $(C_3)$ and $(C_4)$, if $\dsp \frac{H^{11}}{H_1}> \frac{H^{12}}{H_2}$,  we assume that $f$ is concave. Furthermore, assume that the function $h$ is non-increasing on $[0,1]$.
Let $(Z_n)_{n\ge 0}$ be defined as (\ref{a-1}). For any $\varepsilon>0 $, there exists a positive constant $a$ which is independent to $n$, such that for all $n$ large enough, we have
$$
\pp\Big(\left|Z_{n+1}-y^* \right|>\varepsilon\Big)\le  2 e^{-a n},
$$
where $y^{*}$ is the unique equilibrium point of $h$ defined in (\ref{c-1}).
\end{thm}

\begin{rem} From Proposition \ref{prop-i1}, we know that $h$ has a unique solution in $I^{*}\subset[0,1]$.
The condition, that the function $h$ is non-increasing on $[0,1]$, is to guarantee that $h$ has a unique solution in $[0,1]$.
\end{rem}

\begin{rem} Since the function $f$ is differentiable, then it is easy to see that the function $h$ is also differentiable. Through a simple calculation, we get
$$
\aligned
h^{'}(y)=&\frac{1}{[f(y)+f(1-y)]^2}\left\{\Big[-H_1f(y)-H_2f(1-y)\Big]\Big[f(y)+f(1-y)\Big] \right. \\
   &\ \ \left.  +\Big[f^{'}(y)f(1-y)+f(y)f^{'}(1-y)\Big]\Big[(H^{11}-yH_1)-(H^{12}-yH_2)\Big]\right\}.
\endaligned
$$
By noting that $f$ and $f^{'}$ are non-negative function on $[0,1]$, if we take $H^{11}< H^{12}$ and $H^{22}<H^{21}$, then
$$
(H^{11}-yH_1)-(H^{12}-yH_2)=(H^{11}-H^{12})(1-y)+(H^{22}-H^{21})y<0,
$$
which implies that $h^{'}\le 0$ and the function $h$ is non-increasing on $[0,1]$.
\end{rem}

\begin{rem} We consider a linear model with $f(y)=y$ for all $y\in [0,1]$, which satisfies the condition $(C_1)$. Then we have
$$
\aligned
h(y)=&\frac{(H^{11}-yH_1)f(y)+(H^{12}-yH_2)f(1-y)}{f(y)+f(1-y)}\\
=&(H^{11}-yH_1)y+(H^{12}-yH_2)(1-y)\\
=&(H_2-H_1)y^2+(H^{11}-H^{12}-H_2)y+H^{12}
\endaligned
$$
and
$$
h^{'}(y)=2(H_2-H_1)y+(H^{11}-H^{12}-H_2).
$$
Since the function $h^{'}$ is a straight line, in order to make $h^{'}(y)\le 0$ on $[0,1]$, it is enough to check that the following two conditions hold:
$$
h^{'}(0)=H^{11}-H^{12}-H_2\le 0
$$
and
$$
h^{'}(1)=2(H_2-H_1)+(H^{11}-H^{12}-H_2)\le 0,
$$
which are equivalent to
$$
H^{22}\le H^{11}+2H^{21}\ \ \text{and}\ \ \ H^{11}\le H^{2 2}+2H^{12}.
$$
Hence we can take the replacement matrixes $H$ as follows:
$$
H= \begin{pmatrix}
2 & 4\\
3 & 6
\end{pmatrix}
,\ \ \ \
H= \begin{pmatrix}
4 & 1\\
5 & 4
\end{pmatrix}.
$$
\end{rem}

\begin{rem} We consider a nonlinear model with $f(y)=y^2$ for all $y\in [0,1]$, which satisfies the condition $(C_1)$. Then we have
$$
\aligned
h(y)=&\frac{(H^{11}-yH_1)f(y)+(H^{12}-yH_2)f(1-y)}{f(y)+f(1-y)}\\
=&\frac{(H^{11}-yH_1)y^2+(H^{12}-yH_2)(1-y)^2}{y^2+(1-y)^2}\\
=&\frac{-(H_1+H_2)y^3+(H^{11}+H^{12}+2H_2)y^2-(2H^{12}+H_2)y+H^{12}}{2y^2-2y+1}=:\frac{F(y)}{G(y)}
\endaligned
$$
and
$$
h^{'}(y)=\frac{F^{'}(y)G(y)-F(y)G^{'}(y)}{G^2(y)}.
$$
Through simple calculations, we get
$$
\aligned
 P(y):=&F^{'}(y)G(y)-F(y)G^{'}(y)\\
 =&2y(1-y)(H^{11}-yH_1-H^{12}+yH_2)-(2y^2-2y+1)[H_1y^2+H_2(1-y)^2]\\
 =&: P_1(y)-P_2(y).
\endaligned
$$
It is easy to check that for all $y\in [0,1]$,
$$
P_2(y)\ge 0,\ \ \ 2y^2-2y+1\ge \frac{1}{2},\ \ \ 2y(1-y)\le \frac{1}{2}
$$
so in order to make $P(y)\le 0$, it is enough to show that the following claim holds:
\beq\label{r3-1}
\aligned
P_3(y):=&H_1y^2+H_2(1-y)^2-(H^{11}-yH_1-H^{12}+yH_2)\\
=&(H_1+H_2)y^2+(H_1-3H_2)y+H_2+H^{12}-H^{11}\ge 0.
\endaligned
\deq
Let $\dsp y_0:=\frac{3H_2-H_1}{2(H_1+H_2)}$, then $P_3(y)$ can take the minimum point at $y_0$, i.e.,
$$
P_3(y_0)=-\frac{(3H_2-H_1)^2}{4(H_1+H_2)}+H_2+H^{12}-H^{11}.
$$

{\bf Discussion 1:} A direct sufficient condition can be obtained by showing
$$
P_3(0)\ge 0,\ \ P_3(1)\ge 0\ \ \text{and}\ \ P_3(y_0)\ge 0,
$$
which is equivalent to
\beq
H^{11}\le H^{22}+2H^{12},\ \ \ \ H^{22}\le H^{11}+2H^{21}
\deq
and
\beq
10H_1H_2+4(H^{12}-H^{11})(H_1+H_2)\ge 5H_2^2+H_1^2.
\deq
For example, we can take the replacement matrixes $H$ as follows:
$$
H= \begin{pmatrix}
1 & 2\\
3 & 4
\end{pmatrix},\ \ \ \
H= \begin{pmatrix}
4 & 5\\
5 & 6
\end{pmatrix}.
$$

{\bf Discussion 2:} We will discuss the following three cases based on the position of $y_0$.

{\it Case 1: $y_0\le 0$.} For this case, if we have $P_3(0)\ge 0$, then $P_3(y)\ge 0$ for all $y\in[0,1]$. Hence we assume that
$$
3H_2\le H_1\ \ \ \text{and}\ \ \ H^{11}\le H^{22}+2H^{12},
$$
then $P_3(y)\ge 0$ for all $y\in[0,1]$. For example, we can take the replacement matrix $H$ as follows:
$$
H= \begin{pmatrix}
6 & 2\\
10 & 3
\end{pmatrix}.
$$

{\it Case 2: $y_0\ge 1$.} For this case, if we have $P_3(1)\ge 0$, then $P_3(y)\ge 0$ for all $y\in[0,1]$. Hence we assume that
$$
H_2\ge 3H_1\ \ \ \text{and}\ \ \  H^{22}\le H^{11}+2H^{21},
$$
then $P_3(y)\ge 0$ for all $y\in[0,1]$. For example, we can take the replacement matrix $H$ as follows:
$$
H= \begin{pmatrix}
2 & 10\\
3 & 6
\end{pmatrix}.
$$

{\it Case 3: $y_0\in (0,1)$.} For this case, if we have $P_3(y_0)\ge 0$, then $P_3(y)\ge 0$ for all $y\in[0,1]$. Hence we assume that
$$
\frac{1}{3}H_1\le H_2\le 3H_1
$$
and
$$
10H_1H_2+4(H^{12}-H^{11})(H_1+H_2)\ge 5H_2^2+H_1^2,
$$
then $P_3(y)\ge 0$ for all $y\in[0,1]$. For example, we can take the replacement matrixes $H$ as follows:
$$
H= \begin{pmatrix}
2 & 10\\
3 & 4
\end{pmatrix},\ \ \ \
H= \begin{pmatrix}
4 & 2\\
10 & 3
\end{pmatrix}.
$$

Obviously, the conditions in {\bf Discussion 1} are stronger than ones in {\bf Discussion 2}.
Since the function $f(y)=y^2$ is convex, the condition $\dsp \frac{H^{11}}{H_1}\le  \frac{H^{12}}{H_2}$ should be supposed. Through simple testing, the above examples satisfy the condition.
\end{rem}

\subsection{Stochastic approximation algorithm} A stochastic approximation algorithm $(X_n)_{n\ge 0}$ is a stochastic process taking values in $\rr$ and adapted to the filtration $(\FF_n)_{n\ge 1}$ that satisfies
\beq\label{sa-0}
X_{n+1}-X_n=\gamma_{n+1}\big[g(X_n)+U_{n+1}\big],\  \ X_0=x_0\in\rr,
\deq
where $\gamma_n, U_n\in\FF_{n}$, $g:[0,1]\to\rr$ is a function. By using the similar method as Theorem \ref{thm2-1}, we can establish the following large deviation inequality for the stochastic approximation algorithm (\ref{sa-0}).

\begin{thm}\label{thm-sa1} Let $(X_n)_{n\ge 0}$ be a stochastic process taking values in $[0,1]$ and satisfy the stochastic approximation algorithm  (\ref{sa-0}). Assume that the following conditions hold a.s.:
\begin{enumerate}[$(1)$]
 \item for any $n$, we have
\beq\label{sa-1}
\frac{u_l}{n}\leq \gamma_n\leq\frac{u_u}{n},
\deq

\item for any $x\in[0,1]$, we have
\beq\label{sa-2}
|g(x)|\leq K_g,
\deq

\item for any $n$, we have
\beq\label{sa-3}
|U_{n+1}|\le K_u,
\deq

\item for any $n$, we have
\beq\label{sa-4}
\left|\ee\left(\gamma_{n+1}U_{n+1}\  \Big|\ \FF_n\right)\right|\le \frac{K_e}{n^2},
\deq
    \end{enumerate}
    where $u_l, u_u, K_g, K_u, K_e$ are positive constants.
 Furthermore, assume that the function $g$ is non-increasing and continuous on $[0,1]$, and $x^{*}$ is the unique stable point of $g$.
Then for any $\varepsilon>0 $, there exists a positive constant $a$ which is independent to $n$, such that for all $n$ large enough, we have
$$
\pp\Big(\left|X_{n+1}-x^* \right|>\varepsilon\Big)\le  2 e^{-a n}.
$$
\end{thm}

\begin{rem} Under the conditions (\ref{sa-1})-(\ref{sa-4}), Renlund \cite{R-1} proved that $X_n$ converges almost surely to a stable zero of $g$.
\end{rem}

\begin{rem} From the proof of Theorem \ref{thm2-1} (see (\ref{p-6}) and (\ref{p-61})), the assumption for the function $g$ can be weakened as follows. Let $g$ be a non-increasing and bounded function on $[0,1]$, and let $x^*$ be the zero point of $g$, i.e., $g(x^*)=0$. Furthermore, it is assumed that
\beq\label{sa-rl1}
g(x)>0\ \text{for}\ x<x^{*}\ \ \ \text{and}\ \ \ g(x)<0, \ \text{for}\  x>x^{*}.
\deq
\end{rem}

In Theorem \ref{thm-sa1}, we have considered the bounded stochastic approximation algorithm, i.e., $(X_n)_{n\ge 0}$ is a stochastic process taking values in $[0,1]$. Next we shall study the unbounded stochastic approximation algorithm (\ref{sa-0}).

\begin{thm}\label{thm-sa2} Let $(X_n)_{n\ge 0}$ be a stochastic process taking values in $\rr$ and satisfy the stochastic approximation algorithm  (\ref{sa-0}). Suppose that the conditions (\ref{sa-1}), (\ref{sa-3}) and (\ref{sa-4}) hold almost surely.
 Furthermore, assume that the function $g$ is non-increasing, $x^{*}$ is the zero point of $g$ and the condition (\ref{sa-rl1}) is assumed.

   \begin{enumerate}[{\bf (1)}]
 \item For the case $K_{gl}>K_u$, for any $\varepsilon>0 $, there exists a positive constant $a$ which is independent to $n$, such that for all $n$ large enough, we have
$$
\pp\Big(X_{n+1}-x^*>\varepsilon\Big)\le   e^{-a n}.
$$

\item For the case $K_{gl}<K_u$,  for any $\varepsilon>0$ and, for all $n$ large enough, we have
$$
\pp\Big(X_{n+1}-x^*>\varepsilon\Big)\le   \exp\left(-n^{\frac{K_{gl}}{K_u}-\delta}\right)
$$
for any $\delta>0$.

\item For the case $K_{gu}>K_u$, for any $\varepsilon>0 $, there exists a positive constant $a$ which is independent to $n$, such that for all $n$ large enough, we have
$$
\pp\Big(X_{n+1}-x^*<-\varepsilon\Big)\le   e^{-a n}.
$$

\item For the case $K_{gu}<K_u$,  for any $\varepsilon>0$ and, for all $n$ large enough, we have
$$
\pp\Big(X_{n+1}-x^*<-\varepsilon\Big)\le   \exp\left(-n^{\frac{K_{gu}}{K_u}-\delta}\right)
$$
for any $\delta>0$.
  \end{enumerate}
Here $K_{gl}$ and $K_{gu}$ are two positive constants which are defined as
$$
-K_{gl}=\lim_{x\to+\infty} g(x),\ \ \ \ K_{gu}=\lim_{x\to-\infty} g(x).
$$
\end{thm}
\begin{rem}
Under the conditions (\ref{sa-2}) and (\ref{sa-3}), Koml\'os and  R\'ev\'esz \cite{K-R} studied the large deviation inequalities of the stochastic approximation algorithm (\ref{sa-0}) with
$\gamma_n=n^{-1}$ and $(U_n)_{n\ge 1}$ satisfying
$$
\ee (U_{n+1}|U_n, \cdots, U_1)=0\ \ a.s.
$$
Obviously, the conditions (\ref{sa-1}) and (\ref{sa-4}) are weaker than ones in Koml\'os and  R\'ev\'esz \cite{K-R}. In addition, under some appropriate assumptions, Miao and Dong \cite{M-D} established the moderate deviation principle of the stochastic approximation algorithm (\ref{sa-0}), where  $\gamma_n=n^{-1}$ and $(U_n)_{n\ge 1}$ is a sequence of uniformly bounded and independent and identically
distributed random variables.
\end{rem}

\section{Proofs of main results}

Throughout this section, we shall use some notations which have been defined in Section 2, for examples, $(\Delta M_{n})_{n\ge 1}$, $(Z_n)_{n\ge 1}$, $(T_n)_{n\ge 1}$, and so on. Before giving the proof of the main results, some useful preliminary lemmas need to be established firstly.

\begin{lem}\label{lem-1}
For any given $\beta>1$ and for any $\varepsilon>0$ such that
$$
0< \varepsilon \le \frac{4\beta (H_1+H_2)^2}{\min\{H_1, H_2\}} \log\beta.
$$
Then there exists a positive constant $k_0$, such that for all $k>k_0$ and all $n$ large enough, we have
$$
\pp\left(\left|\sum_{i=k}^n \frac{\Delta M_{i+1}}{T_{i+1}}\right|\ge \varepsilon\right)\le 2\exp\left(-\frac{\Big(\varepsilon-\frac{K}{4(\min\{H_1, H_2\})^2}\sum_{i=k}^n \frac{1}{i^2}\Big)^2}{8\frac{\beta (H_1+H_2)^2}{(\min\{H_1, H_2\})^2}\sum_{i=k}^n \frac{1}{i^2}}\right).
$$
where $K=K(H_1,H_2)$ is defined in (\ref{b-6}). In particular, there exist positive constants $k_0^{'}$ and $a=a(\varepsilon, \beta, H_1, H_2)$,  such that for all $k>k_0^{'}$  and all $n$ large enough, we have
 $$
\pp\left(\left|\sum_{i=k}^n \frac{\Delta M_{i+1}}{T_{i+1}}\right|\ge\varepsilon\right)\le 2e^{-ak}.
$$
\end{lem}

\begin{proof}
We only need to prove the form $\dsp\pp\left(\sum_{i=k}^n \frac{\Delta M_{i+1}}{T_{i+1}}\ge \varepsilon\right)$, and the other form $\dsp\pp\left(\sum_{i=k}^n \frac{\Delta M_{i+1}}{T_{i+1}}\le -\varepsilon\right)$ is similar. For $\beta>1$, it is easy to check that
\beq\label{e}
e^x\le1+x+\beta\frac{x^2}{2}\ \ \mbox{for} \  x\le \log \beta.
\deq
Let $A:=4(H_1+H_2)$ and $B:=2\min\{H_1, H_2\}$, then from (\ref{b-1})and (\ref{b-3}), we have
\beq\label{e1}
\frac{1}{T_n}\le \frac{1}{Bn}\ \ \text{and}\ \ \  |\Delta M_{n+1}|\le A.
\deq
By taking $t>0$ such that
$$
\frac{tA}{kB}\le \log \beta,
$$
then for every $i=k, k+1, \cdots, n+1$, we have
\beq\label{e2}
t\left|\frac{\Delta M_{i+1}}{T_{i+1}}\right|\le \frac{tA}{2i\min\{H_1, H_2\}}\le \log \beta.
\deq
From (\ref{e}), (\ref{b-6}) and (\ref{e1}), we get
\begin{align*}
\ee\left[\exp\left(t\frac{\Delta M_{n+1}}{T_{n+1}}\right)\Bigg|\FF_n\right]\le &\ee\left[1+t\frac{\Delta M_{n+1}}{T_{n+1}}+\frac{\beta}{2}t^2\left(\frac{\Delta M_{n+1}}{T_{n+1}}\right)^2\Bigg|\FF_n\right]\\
\le & 1+\frac{Kt}{B^2n^2}+\frac{\beta}{2}\frac{A^2 t^2}{B^2n^2}\\
\le& \exp\left(\frac{Kt}{B^2n^2}+\frac{\beta}{2}\frac{A^2 t^2}{B^2n^2}\right).
\end{align*}
Hence, for $\varepsilon>0$, we have
\begin{align*}
 \pp\left(\sum_{i=k}^n \frac{\Delta M_{i+1}}{T_{i+1}}\ge \varepsilon\right)
=&\pp\left(\exp\left(t\sum_{i=k}^n \frac{\Delta M_{i+1}}{T_{i+1}}\right)
\ge e^{t\varepsilon}\right)\\
\le& e^{-t\varepsilon}\ee\left[\prod_{i=k}^n \exp\left(t\frac{\Delta M_{i+1}}{T_{i+1}}\right)\right]\\
=&e^{-t\varepsilon}\ee\left[\prod_{i=k}^{n-1} \exp\left(t\frac{\Delta M_{i+1}}{T_{i+1}}\right)\ee\left(\exp\left(t\frac{\Delta M_{n+1}}{T_{n+1}}\right)\Bigg|\FF_n\right)\right]\\
\le & \exp\left(-t\varepsilon+\left(\frac{Kt}{B^2}+\frac{\beta}{2}\frac{A^2 t^2}{B^2}\right)\sum_{i=k}^n\frac{1}{i^2}\right).
\end{align*}
For any $\dsp 0< \varepsilon \le \frac{4\beta (H_1+H_2)^2}{\min\{H_1, H_2\}} \log\beta$, it is easy to check that there exists a positive constant $k_1$ such that for all $k>k_1$, we have
$$
\varepsilon>\frac{K}{B^2}\sum_{i=k}^n \frac{1}{i^2},
$$
and there exists another positive constant $k_2$ such that for all $k>k_2$ and all $n$ large enough, we have
$$
\frac{\varepsilon-\frac{K}{B^2}\sum_{i=k}^n \frac{1}{i^2}}{\frac{\beta A^2}{B^2}\sum_{i=k}^n \frac{1}{i^2}}\frac{A}{Bk}\le \log\beta.
$$
Therefore, for all $k>\max\{k_1, k_2\}$, we can choose
$$
t=\frac{\varepsilon-\frac{K}{B^2}\sum_{i=k}^n \frac{1}{i^2}}{\frac{\beta A^2}{B^2}\sum_{i=k}^n \frac{1}{i^2}},
$$
such that for all $n$ large enough
\beq\label{p-7}
\pp\left(\sum_{i=k}^n \frac{\Delta M_{i+1}}{T_{i+1}}\ge \varepsilon\right)\le \exp\left(-\frac{\Big(\varepsilon-\frac{K}{B^2}\sum_{i=k}^n \frac{1}{i^2}\Big)^2}{2\frac{\beta A^2}{B^2}\sum_{i=k}^n \frac{1}{i^2}}\right).
\deq
So the desired results can be obtained.
\end{proof}

\begin{lem}\label{lem-2}
For any $\varepsilon>0$, then for any $\dsp k>\frac{12(H_1+H_2)}{\varepsilon\min\{H_1, H_2\}}$,
 we have
$$
\left\{Z_k-y^{*}\le \frac{1}{2} \varepsilon\right\}\subset \left\{Z_{k+1}-y^{*}\le\frac{3}{4}\varepsilon\right\}
$$
and for any $\dsp k>\frac{12(H_1+H_2)}{\varepsilon\max\{H_1, H_2\}}$
we have
$$
\left\{Z_k-y^{*}\ge- \frac{1}{2} \varepsilon\right\}\subset \left\{Z_{k+1}-y^{*}\ge-\frac{3}{4}\varepsilon\right\}.
$$
\end{lem}

\begin{proof}
From (\ref{a-1}), (\ref{b-1}), (\ref{b-2}) and (\ref{b-3}), we know that
$$
\aligned
Z_{k+1}-y^{*}=&Z_k-y^{*}+\frac{1}{T_{k+1}} (h(Z_k)+\Delta M_{k+1})\\
\le& Z_k-y^{*}+\frac{3(H_1+H_2)}{(k+1)\min\{H_1, H_2\}}
\endaligned
$$
and
 $$
\aligned
Z_{k+1}-y^{*}=&Z_k-y^{*}+\frac{1}{T_{k+1}} (h(Z_k)+\Delta M_{k+1})\\
\ge& Z_k-y^{*}-\frac{3(H_1+H_2)}{(k+1)\max\{H_1, H_2\}},
\endaligned
$$
which implies the lemma.
\end{proof}

\begin{proof}[{\bf Proof of Theorem \ref{thm2-1}}] {\bf Step 1:} Firstly, we shall give the inequality of $\dsp\pp\Big(Z_{n+1}-y^*>\varepsilon\Big)$.
For any $\varepsilon>0$ and $\delta\in(0,1)$, let us define the following events:
$$
B(\varepsilon)=\left\{Z_{[\delta n]}-y^*\ge \frac{\varepsilon}{2}, \cdots, Z_n-y^*\ge\frac{\varepsilon}{2}, Z_{n+1}-y^*\ge \varepsilon\right\},
$$
$$
A^{(n)}_0(\varepsilon)=A_0(\varepsilon)=\left\{Z_1-y^{*}\ge \frac{\varepsilon}{2}, \cdots, Z_n-y^*\ge \frac{\varepsilon}{2}, Z_{n+1}-y^*\ge\varepsilon\right\},
$$
$$
A^{(n)}_k(\varepsilon)=A_k(\varepsilon)=\left\{Z_k-y^*<\frac{\varepsilon}{2}, Z_{k+1}-y^*\ge \frac{\varepsilon}{2}, \cdots, Z_n-y^*\ge \frac{\varepsilon}{2}, Z_{n+1}-y^*\ge\varepsilon\right\},
$$
where $k=1, 2, \cdots, n$ and $[\delta n]$ denotes the integral part of $\delta n$. Hence, for any $F>0$, we can get that
\beq\label{p-2}
\aligned
&\pp\Big(Z_{n+1}-y^*>\varepsilon\Big)\\
=&\pp\Big(Z_{n+1}-y^*>\varepsilon, B(\varepsilon)\Big)+\pp\Big(Z_{n+1}-y^*>\varepsilon,\left(B(\varepsilon)\right)^c\Big)\\
\le& \pp\Big(Z_{[\delta n]}-y^*>F\Big)+\pp\Big(Z_{[\delta n]}-y^*\le F, B(\varepsilon)\Big)+\pp\Big(Z_{n+1}-y^*>\varepsilon,\left(B(\varepsilon)\right)^c\Big).
\endaligned
\deq
From the claim (\ref{l1}) and by using the condition that the function $h$ is non-increasing on $(0,1)$, then for any $\varepsilon>0$ small enough, we have $\dsp h\left(\frac{\varepsilon}{2}+y^*\right)<0$. Hence for any $F>0$, from (\ref{b-1}), we can choose $\delta=\delta(\varepsilon, F, y^*, H_1, H_2, h)$, such that for all $n$ large enough,
\beq\label{p-6}
F+\sum_{i=[\delta n]}^n \frac{h(\frac{\varepsilon}{2}+y^*)}{T_{i+1}}\le F+\frac{h\left(\frac{\varepsilon}{2}+y^*\right)}{2\max\{H_1,H_2\}}\sum_{i=[\delta n]}^n \frac{1}{i+1}\le-\varepsilon.
\deq
From (\ref{a-1}), (\ref{p-6}) and Lemma \ref{lem-1}, there exists a positive constant $a_1$, such that for all $n$ large enough, we have
\begin{align}\label{p-3}
&\pp\Big(Z_{[\delta n]}-y^*\le F, B(\varepsilon)\Big)\nonumber\\
=&\pp\left(Z_{[\delta n]}-y^*\le F, Z_{[\delta n]}-y^*\ge \frac{\varepsilon}{2}, \cdots, Z_n-y^*\ge\frac{\varepsilon}{2}, Z_{n+1}-y^*\ge \varepsilon \right)\nonumber\\
=& \pp\Bigg(\varepsilon\le Z_{n+1}-y^*=Z_{[\delta n]}-y^*+\sum_{i=[\delta n]}^n \frac{h(Z_i)+\Delta M_{i+1}}{T_{i+1}},\nonumber\\
&\ \ \ \ \ \ \ \ \ \ \ \ \ \ \ \ \ Z_{[\delta n]}-y^*\le F, Z_{[\delta n]}-y^*\ge \frac{\varepsilon}{2}, \cdots, Z_n-y^*\ge\frac{\varepsilon}{2}\Bigg)\nonumber\\
\le& \pp\Bigg(\varepsilon\le Z_{n+1}-y^*=Z_{[\delta n]}-y^*+\sum_{i=[\delta n]}^n \frac{h\left(\frac{\varepsilon}{2}+y^*\right)+\Delta M_{i+1}}{T_{i+1}},\nonumber\\
&\ \ \ \ \ \ \ \ \ \ \ \ \ \ \ \ \ Z_{[\delta n]}-y^*\le F, Z_{[\delta n]}-y^*\ge \frac{\varepsilon}{2}, \cdots, Z_n-y^*\ge\frac{\varepsilon}{2}\Bigg)\nonumber\\
\le& \pp\left(\varepsilon\le  F+\sum_{i=[\delta n]}^n \frac{h(\frac{\varepsilon}{2}+y^*)+\Delta M_{i+1}}{T_{i+1}}\right)\nonumber\\
\le& \pp\left(\sum_{i=[\delta n]}^n \frac{\Delta M_{i+1}}{T_{i+1}}\ge 2\varepsilon\right) \le e^{-a_1\delta n}.
\end{align}
Next we estimate the probability $\pp\Big(Z_{n+1}-y^*>\varepsilon,\left(B(\varepsilon)\right)^c\Big)$. Firstly, we know that
\begin{align*}
\pp\Big(Z_{n+1}-y^*>\varepsilon,\left(B(\varepsilon)\right)^c\Big)\le \sum_{k=[\delta n]}^n \pp(A_k(\varepsilon)).
\end{align*}
 By using the similar method as (\ref{p-6}), for all $n$ large enough and for every $k\ge [\delta n]$, we have
\beq\label{p-61}
\sum_{i=k+1}^n \frac{h\big(\frac{\varepsilon}{2}+y^*\big)}{T_{i+1}}\le \frac{h\left(\frac{\varepsilon}{2}+y^*\right)}{2\max\{H_1,H_2\}}\sum_{i=k+1}^n \frac{1}{i+1}\le-\frac{3\varepsilon}{4}.
\deq
From (\ref{a-1}), (\ref{p-61}), Lemma \ref{lem-1} and Lemma \ref{lem-2}, there exists a positive constant $a_2$, such that for all $n$ large enough, we have
\begin{align*}
\pp(A_k(\varepsilon))&=\pp\left(Z_k-y^*<\frac{\varepsilon}{2}, Z_{k+1}-y^*\ge \frac{\varepsilon}{2}, \cdots, Z_n-y^*\ge \frac{\varepsilon}{2}, Z_{n+1}-y^*\ge\varepsilon\right)\\
&=\pp\Bigg(Z_k-y^*<\frac{\varepsilon}{2},  Z_{k+1}-y^*+\sum_{i=k+1}^n \frac{h(Z_i)+\Delta M_{i+1}}{T_{i+1}} \ge\varepsilon,\ \\
&\ \ \ \ \ \ \ \ \ \ \ \ \ \ \ \   Z_{k+1}-y^*\ge \frac{\varepsilon}{2}, \cdots, Z_n-y^*\ge \frac{\varepsilon}{2}\Bigg)\\
&\le \pp\Bigg(Z_k-y^*<\frac{\varepsilon}{2},  Z_{k+1}-y^*+\sum_{i=k+1}^n \frac{h\big(\frac{\varepsilon}{2}+y^*\big)+\Delta M_{i+1}}{T_{i+1}} \ge\varepsilon,\ \\
&\ \ \ \ \ \ \ \ \ \ \ \ \ \ \ \   Z_{k+1}-y^*\ge \frac{\varepsilon}{2}, \cdots, Z_n-y^*\ge \frac{\varepsilon}{2}\Bigg)\\
&\le \pp\left(Z_k-y^*<\frac{\varepsilon}{2}, Z_{k+1}-y^*+\sum_{i=k+1}^n \frac{\Delta M_{i+1}}{T_{i+1}} \ge \frac{7\varepsilon}{4} \right)\\
&\le \pp\left(Z_{k+1}-y^{*}\le\frac{3}{4}\varepsilon, Z_{k+1}-y^*+\sum_{i=k+1}^n \frac{\Delta M_{i+1}}{T_{i+1}} \ge\frac{7\varepsilon}{4}\right)\\
&\le \pp\left( \sum_{i=k+1}^n \frac{\Delta M_{i+1}}{T_{i+1}}\ge \varepsilon\right)\le  e^{-a_2(k+1)}.
\end{align*}
Thus there exists a positive constant $a_3$, such that for all $n$ large enough, we have
\beq\label{p-4}
\pp\Big(Z_{n+1}-y^*>\varepsilon,\left(B(\varepsilon)\right)^c\Big)\le \sum_{k=[\delta n]}^n e^{-a_2(k+1)}\le  e^{-a_3 \delta n}.
\deq
At last, we estimate the probability $\pp\Big(Z_{[\delta n]}-y^*>F\Big)$. From the definition of $Z_n$, we know that $|Z_n|\le 1$. Hence, by taking $F>1+y^*$, we have
\begin{align}\label{p-5}
\pp\Big(Z_{[\delta n]}-y^*>F\Big)=0.
\end{align}
From (\ref{p-2}), (\ref{p-3}), (\ref{p-4}) and (\ref{p-5}), there exists a positive constant $a_4$, such that
$$
\pp\Big(Z_{n+1}-y^*>\varepsilon\Big)\le  e^{-a_4 n}.
$$

{\bf Step 2:} Next we shall obtain the inequality of $\dsp\pp\Big(Z_{n+1}-y^*<-\varepsilon\Big)$.
For any $\varepsilon>0$ and $\delta\in(0,1)$, let us define the following events:
$$
C(\varepsilon)=\left\{Z_{[\delta n]}-y^*\le -\frac{\varepsilon}{2}, \cdots, Z_n-y^*\le-\frac{\varepsilon}{2}, Z_{n+1}-y^*\le -\varepsilon\right\},
$$
$$
D^{(n)}_0(\varepsilon)=D_0(\varepsilon)=\left\{Z_1-y^{*}\le -\frac{\varepsilon}{2}, \cdots, Z_n-y^*\le -\frac{\varepsilon}{2}, Z_{n+1}-y^*\le-\varepsilon\right\},
$$
$$
\aligned
D^{(n)}_k(\varepsilon)=&D_k(\varepsilon)\\
=&\left\{Z_k-y^*>-\frac{\varepsilon}{2}, Z_{k+1}-y^*\le -\frac{\varepsilon}{2}, \cdots, Z_n-y^*\le -\frac{\varepsilon}{2}, Z_{n+1}-y^*\le-\varepsilon\right\},
\endaligned
$$
where $k=1, 2, \cdots, n$. Hence, for any $F>0$, we can get that
\beq\label{p-2-1}
\aligned
&\pp\Big(Z_{n+1}-y^*<-\varepsilon\Big)\\
=&\pp\Big(Z_{n+1}-y^*<-\varepsilon, C(\varepsilon)\Big)+\pp\Big(Z_{n+1}-y^*<-\varepsilon,\left(C(\varepsilon)\right)^c\Big)\\
\le& \pp\Big(Z_{[\delta n]}-y^*<-F\Big)+\pp\Big(Z_{[\delta n]}-y^*\ge- F, C(\varepsilon)\Big)+\pp\Big(Z_{n+1}-y^*<-\varepsilon,\left(C(\varepsilon)\right)^c\Big).
\endaligned
\deq
From the claim (\ref{l1}) and by using the condition that the function $h$ is non-increasing on $(0,1)$, then for any $\varepsilon>0$ small enough, we have $\dsp h\left(y^*-\frac{\varepsilon}{2}\right)>0$. Hence for any $F>0$, from (\ref{b-1}), we can choose $\delta=\delta(\varepsilon, F, y^*, H_1, H_2, h)$, such that for all $n$ large enough,
\beq\label{p-6-1}
F-\sum_{i=[\delta n]}^n \frac{h\big(y^*-\frac{\varepsilon}{2}\big)}{T_{i+1}}\le F-\frac{h\left(y^*-\frac{\varepsilon}{2}\right)}{2\max\{H_1,H_2\}}\sum_{i=[\delta n]}^n \frac{1}{i+1}\le-\varepsilon.
\deq
From (\ref{a-1}), (\ref{p-6-1}) and Lemma \ref{lem-1}, there exists a positive constant $a_1$, such that for all $n$ large enough, we have
\begin{align}\label{p-3-1}
&\pp\Big(Z_{[\delta n]}-y^*\ge -F, C(\varepsilon)\Big)\nonumber\\
=&\pp\left(Z_{[\delta n]}-y^*\ge -F, Z_{[\delta n]}-y^*\le -\frac{\varepsilon}{2}, \cdots, Z_n-y^*\le-\frac{\varepsilon}{2}, Z_{n+1}-y^*\le -\varepsilon \right)\nonumber\\
=& \pp\Bigg( Z_{n+1}-y^*=Z_{[\delta n]}-y^*+\sum_{i=[\delta n]}^n \frac{h(Z_i)+\Delta M_{i+1}}{T_{i+1}}\le-\varepsilon,\nonumber\\
&\ \ \ \ \ \ \ \ \ \ \ \ \ \ \ \ \ Z_{[\delta n]}-y^*\ge -F, Z_{[\delta n]}-y^*\le -\frac{\varepsilon}{2}, \cdots, Z_n-y^*\le-\frac{\varepsilon}{2}\Bigg)\nonumber\\
\le& \pp\Bigg( Z_{n+1}-y^*=Z_{[\delta n]}-y^*+\sum_{i=[\delta n]}^n \frac{h\left(y^*-\frac{\varepsilon}{2}\right)+\Delta M_{i+1}}{T_{i+1}}\le -\varepsilon,\nonumber\\
&\ \ \ \ \ \ \ \ \ \ \ \ \ \ \ \ \ Z_{[\delta n]}-y^*\ge-F, Z_{[\delta n]}-y^*\le -\frac{\varepsilon}{2}, \cdots, Z_n-y^*\le-\frac{\varepsilon}{2}\Bigg)\nonumber\\
\le& \pp\left( -F+\sum_{i=[\delta n]}^n \frac{h(y^*-\frac{\varepsilon}{2})+\Delta M_{i+1}}{T_{i+1}}\le-\varepsilon \right)\nonumber\\
\le& \pp\left(\sum_{i=[\delta n]}^n \frac{\Delta M_{i+1}}{T_{i+1}}\le -2\varepsilon\right) \le e^{-a_1\delta n}.
\end{align}
Next we estimate the probability $\pp\Big(Z_{n+1}-y^*<-\varepsilon,\left(C(\varepsilon)\right)^c\Big)$. Firstly, we know that
\begin{align*}
\pp\Big(Z_{n+1}-y^*<-\varepsilon,\left(C(\varepsilon)\right)^c\Big)\le \sum_{k=[\delta n]}^n \pp(D_k(\varepsilon)).
\end{align*}
 By using the similar method as (\ref{p-6-1}), for all $n$ large enough and for every $k\ge [\delta n]$, we have
\beq\label{p-61-1}
\sum_{i=k+1}^n \frac{h\big(y^*-\frac{\varepsilon}{2}\big)}{T_{i+1}}\ge \frac{h\left(y^*-\frac{\varepsilon}{2}\right)}{2\max\{H_1,H_2\}}\sum_{i=k+1}^n \frac{1}{i+1}\ge\frac{3\varepsilon}{4}.
\deq
From (\ref{a-1}), (\ref{p-61-1}), Lemma \ref{lem-1} and Lemma \ref{lem-2}, there exists a positive constant $a_2$, such that for all $n$ large enough, we have
\begin{align*}
\pp(D_k(\varepsilon))&=\pp\left(Z_k-y^*>-\frac{\varepsilon}{2}, Z_{k+1}-y^*\le -\frac{\varepsilon}{2}, \cdots, Z_n-y^*\le -\frac{\varepsilon}{2}, Z_{n+1}-y^*\le-\varepsilon\right)\\
&=\pp\Bigg(Z_k-y^*>-\frac{\varepsilon}{2},  Z_{k+1}-y^*+\sum_{i=k+1}^n \frac{h(Z_i)+\Delta M_{i+1}}{T_{i+1}} \le-\varepsilon,\ \\
&\ \ \ \ \ \ \ \ \ \ \ \ \ \ \ \   Z_{k+1}-y^*\le -\frac{\varepsilon}{2}, \cdots, Z_n-y^*\le -\frac{\varepsilon}{2}\Bigg)\\
&\le \pp\Bigg(Z_k-y^*>-\frac{\varepsilon}{2},  Z_{k+1}-y^*+\sum_{i=k+1}^n \frac{h\big(y^*-\frac{\varepsilon}{2}\big)+\Delta M_{i+1}}{T_{i+1}} \le-\varepsilon,\ \\
&\ \ \ \ \ \ \ \ \ \ \ \ \ \ \ \   Z_{k+1}-y^*\le -\frac{\varepsilon}{2}, \cdots, Z_n-y^*\le -\frac{\varepsilon}{2}\Bigg)\\
&\le \pp\left(Z_k-y^*>-\frac{\varepsilon}{2}, Z_{k+1}-y^*+\sum_{i=k+1}^n \frac{\Delta M_{i+1}}{T_{i+1}} \le -\frac{7\varepsilon}{4} \right)\\
&\le \pp\left(Z_{k+1}-y^{*}\ge-\frac{3}{4}\varepsilon, Z_{k+1}-y^*+\sum_{i=k+1}^n \frac{\Delta M_{i+1}}{T_{i+1}} \le-\frac{7\varepsilon}{4}\right)\\
&\le \pp\left( \sum_{i=k+1}^n \frac{\Delta M_{i+1}}{T_{i+1}}\le -\varepsilon\right)\le  e^{-a_2(k+1)}.
\end{align*}
Thus there exists a positive constant $a_3$, such that for all $n$ large enough, we have
\beq\label{p-4-1}
\pp\Big(Z_{n+1}-y^*<-\varepsilon,\left(C(\varepsilon)\right)^c\Big)\le \sum_{k=[\delta n]}^n e^{-a_2(k+1)}\le  e^{-a_3 \delta n}.
\deq
At last, we estimate the probability $\pp\Big(Z_{[\delta n]}-y^*<-F\Big)$. From the definition of $Z_n$, we know that $|Z_n|\le 1$. Hence, by taking $F>1+y^*$, we have
\begin{align}\label{p-5-1}
\pp\Big(Z_{[\delta n]}-y^*<-F\Big)=0.
\end{align}
From (\ref{p-2-1}), (\ref{p-3-1}), (\ref{p-4-1}) and (\ref{p-5-1}), there exists a positive constant $a_4$, such that
$$
\pp\Big(Z_{n+1}-y^*<-\varepsilon\Big)\le  e^{-a_4 n}.
$$
\end{proof}

\begin{proof}[{\bf Proof of Theorem \ref{thm-sa2}}] By using the similar method as the proof of Lemma \ref{lem-1}, for any $\varepsilon>0$, there exist positive constants $k_1$ and $a$, such that for all $k>k_1$  and all $n$ large enough, we have
 \beq\label{thm-sa2-p1}
\pp\left(\left|\sum_{i=k}^n \gamma_{i+1}U_{i+1}\right|\ge\varepsilon\right)\le 2e^{-ak}.
\deq
From the proof of Lemma \ref{lem-2}, for any $\varepsilon>0$, there exists a positive constant $k_2$, such that for all $k>k_2$, we have
\beq\label{thm-sa2-p2}
\left\{X_k-x^{*}\le \frac{1}{2} \varepsilon\right\}\subset \left\{X_{k+1}-x^{*}\le\frac{3}{4}\varepsilon\right\}
\deq
and
\beq\label{thm-sa2-p3}
\left\{X_k-x^{*}\ge- \frac{1}{2} \varepsilon\right\}\subset \left\{X_{k+1}-x^{*}\ge-\frac{3}{4}\varepsilon\right\}.
\deq
From (\ref{p-2}) and the similar proof of Theorem \ref{thm2-1}, for any $\varepsilon>0$, $\delta\in(0,1)$ and $F>0$, there exists a positive constant $a_1$, such that for all $n$ large enough, we have
\beq\label{thm-sa2-p4}
\aligned
\pp\Big(X_{n+1}-x^*>\varepsilon\Big)
\le \pp\Big(X_{[\delta n]}-x^*>F\Big)+ e^{-a_1n}.
\endaligned
\deq
Hence, in order to obtain the desired result, it is enough to estimate the term
$$\dsp\pp\Big(X_{[\delta n]}-x^*>F\Big).$$
Let $m=[\delta n]$, and for $k=1,2,\cdots,m-1$ let
$$
A^{(m)}_k(F):=\left\{X_k-x^*<\frac{F}{2}, X_{k+1}-x^*\ge \frac{F}{2}, \cdots, X_{m-1}-x^*\ge \frac{F}{2}, X_{m}-x^*\ge F\right\},
$$
and
$$
A^{(m)}_0(F):=\left\{X_1-x^*\ge \frac{F}{2}, X_{k+1}-x^*\ge \frac{F}{2}, \cdots, X_{m-1}-x^*\ge \frac{F}{2}, X_{m}-x^*\ge F\right\}.
$$
Then for $k=1,2,\cdots,m-1$, we have
\beq\label{thm-sa2-p5}
\aligned
&\pp\Big(A^{(m)}_k(F)\Big)\\
=&\pp\left(X_k-x^*<\frac{F}{2}, X_{k+1}-x^*\ge \frac{F}{2}, \cdots, X_{m-1}-x^*\ge \frac{F}{2}, X_{m}-x^*\ge F\right)\\
=&\pp\bigg(X_k-x^*<\frac{F}{2}, X_{k+1}-x^*\ge \frac{F}{2}, \cdots, X_{m-1}-x^*\ge \frac{F}{2}, \\
&\ \ \ \ \ \ \ \ \ \ \ \ \   \ \ \ \ \ \ \  X_{m}-x^*=X_{k+1}-x^*+\sum_{i=k+1}^{m-1}\gamma_{i+1}\big[g(X_{i})+U_{i+1}\big]\ge F\bigg)\\
\le &\pp\left(X_{k+1}-x^*+\sum_{i=k+1}^{m-1}\gamma_{i+1}\left[g\left(x^*+\frac{F}{2}\right)+U_{i+1}\right]\ge F, X_k-x^*<\frac{F}{2}\right).
\endaligned
\deq
On one hand, we have
\beq\label{thm-sa2-p6}
\aligned
\pp\Big(X_{m}-x^*>F\Big)\le\sum_{k=0}^{m-1}\pp\Big(A^{(m)}_k(F)\Big).
\endaligned
\deq
On the other hand, from (\ref{thm-sa2-p2}), there exists a positive constant $k_0:=\max\{k_1, k_2\}$, such that for all $k\ge k_0$, we have
\beq\label{thm-sa2-p7}
\aligned
\pp\Big(A^{(m)}_k(F)\Big)
\le \pp\left(\sum_{i=k+1}^{m-1}\gamma_{i+1}\left[g\left(x^*+\frac{F}{2}\right)+U_{i+1}\right]\ge \frac{F}{4}\right).
\endaligned
\deq

{\bf (1).} From the condition $\dsp\lim_{x\to+\infty}[-g(x)]=K_{gl}>K_u$, there exists a positive contant $F_0$ such that for $F>F_0$, we have
$$
g\left(x^*+\frac{F}{2}\right)+U_{i+1}<0,\ \ a.s.
$$
which, from (\ref{thm-sa2-p7}), implies that for any $k> k_0$, $\pp\Big(A^{(m)}_k(F)\Big)=0$. Furthermore, for $0\le k\le k_0$, it is easy to see that $X_k$ is bounded. Hence we have $\pp\Big(X_{m}-x^*>F\Big)=0$ for all $F$ large enough.

{\bf (2).} By using the inequality (\ref{thm-sa2-p4}), for any $F>0$, $0<\delta<\mu<1$ and $F_1>0$, there exists a positive constant $a_2$, such that for all $n$ large enough, we have
\beq\label{thm-sa2-p8}
\aligned
\pp\Big(X_{m}-x^*>F\Big)
\le \pp\Big(X_{[\mu m]}-x^*>F_1\Big)+ e^{-a_2n}.
\endaligned
\deq
For $k< k_0$, since $|X_k|\le |X_{k-1}|+\gamma_{k+1}|g(X_{k})|+\gamma_{k+1}|U_{k+1}|$, then there exists a positive constant $C$, such that
$|X_k-x^*|\le C$ a.s. for every $0\le k\le k_0$. From (\ref{thm-sa2-p5}), we get
\beq\label{thm-sa2-p9}
\pp\Big(A^{([\mu m])}_k(F_1)\Big)\le \pp\left(\sum_{i=k+2}^{[\mu m]}\frac{u_u}{i}U_{i}\ge F_1-C-u_u g\left(x^*+\frac{F_1}{2}\right)\sum_{i=k+2}^{[\mu m]}\frac{1}{i}\right).
\deq
For $k_0\le k< [\mu m]$, from (\ref{thm-sa2-p7}), we have
\beq\label{thm-sa2-p10}
\pp\Big(A^{([\mu m])}_k(F_1)\Big)\le
\pp\left(\sum_{i=k+2}^{[\mu m]}\frac{u_u}{i}U_{i}\ge \frac{F_1}{4}-u_u g\left(x^*+\frac{F_1}{2}\right)\sum_{i=k+2}^{[\mu m]}\frac{1}{i}\right).
\deq
In order to estimate the probabilities in (\ref{thm-sa2-p9}) and (\ref{thm-sa2-p10}), we need to give an exponential inequality with the following form:
$$
\pp\left(\sum_{i=k}^{n}\frac{U_{i}}{i}\ge B\sum_{i=k}^{n}\frac{1}{i}+D\right),
$$
where $D>0$, $|U_i|\le L$ a.s. for all $1\le i\le n$ and $0\le  B\le L$. Define
$$
k_n=\Big[e^{-4}e^{\frac{D}{L}}k^{1-\frac{B}{L}}n^{\frac{B}{L}}\Big].
$$
For any $\nu\in(0,1)$ small enough (which may depend on $L, B, D$), we can check that $k\le k_n$ for $k\le \nu n$. Hence from (\ref{thm-sa2-p1}) (in where if $\gamma_i$ is replaced by $i$, the inequality (\ref{thm-sa2-p1}) still holds), there exists a positive constant $a_3$, such that
\beq\label{thm-sa2-p11}
\aligned
&\pp\left(\sum_{i=k}^{n}\frac{U_{i}}{i}\ge B\sum_{i=k}^{n}\frac{1}{i}+D\right)\\
\le &\pp\left(\sum_{i=k}^{k_n}\frac{U_{i}}{i}\ge B\sum_{i=k}^{n}\frac{1}{i}+D-2L\right)
+\pp\left(\sum_{i=k_n+1}^{n}\frac{U_{i}}{i}\ge 2L\right)\\
=&\pp\left(\sum_{i=k_n+1}^{n}\frac{U_{i}}{i}\ge 2L\right)\le \exp\left(-a_3k^{1-\frac{B}{L}}n^{\frac{B}{L}}\right).
\endaligned
\deq
From (\ref{thm-sa2-p9}), (\ref{thm-sa2-p10}) and (\ref{thm-sa2-p11}), there exists a positive constant $a_4$, such that
$$
\pp\Big(X_{[\mu m]}-x^*>F_1\Big)\le \sum_{k=0}^{[\mu m]-1}\pp\Big(A^{([\mu m])}_k(F_1)\Big)\le \exp\left(-a_4n^{-\frac{g\left(x^*+\frac{F_1}{2}\right)}{K_u}}\right).
$$
By choosing $F_1$ so large that it satisfies
$$
-\frac{g\left(x^*+\frac{F_1}{2}\right)}{K_u}>\frac{K_{gl}}{K_u}-\delta,
$$
and from (\ref{thm-sa2-p8}), we have for all $n$ large enough
$$
\pp\Big(X_{m}-x^*>F\Big)\le\exp\left(-n^{\frac{K_{gl}}{K_u}-\delta}\right).
$$

By using similar method, the results in {\bf (3)} and {\bf (4)} can be proved.
\end{proof}

\end{document}